\documentclass{amsart}

\theoremstyle{definition}

\theoremstyle{remark}

\begin{document}

\title{The Cauchy-Schl\"Omilch transformation}

\author[T. Amdeberhan]{T. Amdeberhan}
\address{Department of Mathematics,
Tulane University, New Orleans, LA 70118}
\email{tamdeberhan@math.tulane.edu}

\author[M. L. Glasser]{M. L. Glasser}
\address{Department of Mathematics and Computer Science, 
\newline
Clarkson University, Postdam, 
New York 13676}
\email{laryg@clarkson.math.edu}

\author[M. C. Jones]{M. C. Jones}
\address{Department of Mathematics and Statistics,
The Open University, Milton Keynes, United Kingdom}
\email{m.c.jones@open.ac.uk}

\author[V. Moll]{V. H. Moll}
\address{Department of Mathematics,
Tulane University, New Orleans, LA 70118}
\email{vhm@math.tulane.edu}

\author[R. Posey]{R. Posey}
\address{Department of Mathematics,
Baton Rouge Community College, 
\newline
Baton Rouge, LA 70806}
\email{rposey@gmail.com}

\author[D. Varela]{D. Varela}
\address{DVD Academic Consulting}
\email{xentrico@yahoo.com}

\subjclass{Primary 33}

\date{\today}

\keywords{}

\begin{abstract}
The Cauchy-Schl\"omilch transformation states that for a function $f$ and 
$a, \, b > 0$, the 
integral of $f(x^{2})$ and $af((ax-bx^{-1})^{2}$ over the interval
$[0, \infty)$ are the same.   This 
elementary result is used to evaluate 
many non-elementary definite integrals, most of which cannot be obtained 
by symbolic packages. Applications to probability distributions is also 
given. 
\end{abstract}

\maketitle

\newcommand{\nn}{\nonumber}
\newcommand{\ba}{\begin{eqnarray}}
\newcommand{\ea}{\end{eqnarray}}
\newcommand{\dz}{\frac{d}{dz}}
\newcommand{\E}{{\mathfrak{E}}}
\newcommand{\F}{{\mathfrak{F}}}
\newcommand{\Ro}{{\mathfrak{R}}}
\newcommand{\ift}{\int_{0}^{\infty}}
\newcommand{\iftt}{\int_{- \infty}^{\infty}}
\newcommand{\no}{\noindent}
\newcommand{\victor}{{}}
\newcommand{\X}{{\mathbb{X}}}
\newcommand{\Q}{{\mathbb{Q}}}
\newcommand{\R}{{\mathbb{R}}}
\newcommand{\Y}{{\mathbb{Y}}}
\newcommand{\Ftwo}{{{_{2}F_{1}}}}
\newcommand{\realpart}{\mathop{\rm Re}\nolimits}
\newcommand{\imagpart}{\mathop{\rm Im}\nolimits}

\newtheorem{Definition}{\bf Definition}[section]
\newtheorem{Thm}[Definition]{\bf Theorem} 
\newtheorem{Example}[Definition]{\bf Example} 
\newtheorem{Lem}[Definition]{\bf Lemma} 
\newtheorem{Note}[Definition]{\bf Note} 
\newtheorem{Cor}[Definition]{\bf Corollary} 
\newtheorem{Prop}[Definition]{\bf Proposition} 
\newtheorem{Problem}[Definition]{\bf Problem} 
\numberwithin{equation}{section}

\section{Introduction} \label{intro} 
\setcounter{equation}{0}

The problem of analytic evaluations of definite integrals has been of 
interest to scientists for a long time. The central 
question can be stated as follows: \\

\begin{center}
{\em given a class of 
functions} $\mathfrak{F}$ {\em and an 
interval} $[a,b] \subset \mathbb{R}$, {\em express the integral of} $f \in 
\mathfrak{F}$ 
\begin{equation}
I = \int_{a}^{b} f(x) \, dx,
\nonumber
\end{equation}
\noindent
{\em in terms of special values of functions in an enlarged class}
$\mathfrak{G}$. 
\end{center}

\medskip

Many methods for the evaluation of definite integrals have been  developed since
the early stages of Integral Calculus, which 
resulted in a variety of ad-hoc
techniques for producing closed-form expressions. 
Although a general procedure 
applicable to all integrals is undoubtedly unattainable, it is within reason to
expect a systematic cataloguing procedure for large groups of 
definite integrals.  To this effect, one of the authors has instituted a
project to verify all the entries in the popular table 
by I. S. Gradshteyn and I. M. Ryzhik \cite{gr}.  The website
\begin{center}
{\tt{http://www.math.tulane.edu/$\sim$vhm/$\text{web}_{-}\text{html}$/pap-index.html}}
\end{center}
\noindent
contains a series of papers as a treatment to the above-alluded project.

Naturally, any document containing a 
large number of entries, such as the table \cite{gr} or the encyclopedic
treatise \cite{prudnikov1}, is 
likely to contain errors, many of which arising from transcription from 
other tables. The earliest extensive table of integrals still accessible
is \cite{bierens1}, compiled by Bierens de Haan who also presented 
in \cite{bierens2} a survey of the methods employed in the verification of 
the entries from \cite{gr}. These tables form the main source for \cite{gr}.

The revision of integral tables 
is nothing new. C. F. Lindman \cite{lindman1} compiled a long list of 
errors from the table by Bierens de Haan \cite{bierens3}. The editors of 
\cite{gr} maintain the webpage 
\begin{center}
\texttt{http://www.mathtable.com/gr/}
\end{center}
\noindent
where the corrections to the table are 
stored. 

Many techniques have been developed for  evaluating definite 
integrals, and the goal
of this paper is to present one such method 
popularized by Schl\"omilch in \cite{schlomilch1}. 
The  identity (\ref{transf1}) below appeared 
in \cite{irrbook}, where it was called the 
{\em Schl\"omilch transform}, although it was used by 
J. Liouville \cite{liouville1} to evaluate the integral 
\begin{equation}
\int_{0}^{1} \frac{t^{\mu + 1/2} (1-t)^{\mu - 1/2} \, dt}
{(a+bt-ct^{2})^{\mu+1}},
\label{int-1}
\end{equation}
\noindent
and in \cite{liouville2} Liouville  quotes a letter from 
Schl\"omilch in which he describes his approach to (\ref{int-1})
via the formula 
\begin{equation}
\int_{0}^{\infty} F \left( \frac{\alpha}{x} + \gamma x \right) 
\frac{dx}{\sqrt{x}} = \frac{1}{\sqrt{\gamma}} 
\int_{0}^{\infty} F( 2 \sqrt{\alpha 
\gamma} + y ) \frac{dy}{\sqrt{y}}
\victor
\end{equation}
\noindent
to derive the reduction
\begin{equation}
\int_{0}^{\infty} \frac{x^{m+1/2} \, dx}
{( \alpha + \beta x + \gamma x^{2})^{m+1}} = 
\frac{1}{\sqrt{\gamma}} \int_{0}^{\infty} \frac{y^{1/2-m} \, dy }
{(\beta + 2 \sqrt{\alpha \gamma}  + y)^{m+1}}.
\victor
\end{equation}
\noindent
The integral (\ref{int-1}) is then evaluated in terms of 
the {\em beta function}.  
Schl\"omilch 
also states that his method can be found in a note by A. Cauchy 
\cite{cauchy-1823} published some $35$ years earlier\footnote{This note is available at 
\texttt{http://gallica.bnf.fr/ark:/12148/bpt6k90193x.zoom.f526}}. In view of
this historical precedence, the name 
{\em Cauchy-Schl\"omilch} adopted here seems to be more 
justified. Some illustrative examples appear in the text \cite{melzak1}. \\

We present here a variety of definite integrals evaluated by use of the 
Schl\"omilch 
transform and its extensions. Several of the examples are not 
computable by the current symbolic languages. For each evaluation presented 
in the upcoming sections, we 
considered its computation using \texttt{Mathematica}. Naturally, the
statement `the integral can not be computed symbolically' has to be 
complemented with the phrase `at the present date'. 

\section{The Cauchy-Schl\"omilch 
transformation} \label{sec-schlo} 
\setcounter{equation}{0}

In this section we present the basic result accompanied with initial 
examples. Further applications are 
discussed in the remaining sections. 

\begin{Thm}
\label{main-thm}
[Cauchy-Schl\"omilch] Let 
$a, \, b >0$ and assume that $f$ is a continuous function for 
which the integrals in (\ref{transf1}) are 
convergent. Then
\ba
\ift f \left( \left(ax - bx^{-1} \right)^{2} \right) \, dx & = & 
\frac{1}{a} \ift f(y^{2}) \, dy. 
\label{transf1}
\ea
\end{Thm}
\begin{proof}
The change of variables $t = b/ax$ yields
\ba
I & = & \ift f \left( \left(ax - b/x \right)^{2} \right) \, dx \nn \\
& = & \frac{b}{a}  \ift f \left( \left(at - b/t \right)^{2} \right) \, 
t^{-2} \, dt.  \nn
\ea
\no
The average of these two representations, followed by the change of variables 
$u = ax - b/x$ completes the proof.
\end{proof}

The next result is a direct consequence of Theorem \ref{main-thm}.

\begin{Cor}
Preserve the assumptions of Theorem \ref{main-thm}. Let 
$\begin{displaystyle} h_{n} = \sum_{k=0}^{n} c_{k}x^{2k+1}
\end{displaystyle}$  be an odd  polynomial. Define 
\begin{equation}
g_{n}(x) = x \left( \sum_{k=0}^{n} d_{k}x^{k} \right)^{2},
\end{equation}
\noindent
where 
\begin{equation}
d_{k} = \sum_{j=k}^{n} \binom{k+j}{2k} \frac{2j+1}{2k+1} (ab)^{j-k} c_{j}.
\label{form-d}
\end{equation}
\noindent
Then
\begin{equation}
\int_{0}^{\infty} f \left( \left[ h_{n}(ax) - h_{n}(bx^{-1}) \right]^{2}
\right) \, dx = \frac{1}{a} \int_{0}^{\infty} f(g_{n}(y^{2})) \, dy. 
\label{form-d1}
\victor
\end{equation}
\end{Cor}
\begin{proof}
Denote 
\begin{equation}
H_{n}(x) = h_{n}(ax) - h_{n}(bx^{-1}) = \sum_{k=0}^{n} c_{k} \psi_{k}(x)
\end{equation}
\noindent
where $\psi_{k}(x) = (ax)^{2k+1}-(bx^{-1})^{2k+1}$ and let 
$\phi_{k}(x) = (ax-bx^{-1})^{2k+1}$. Then, the polynomials $\psi_{k}$ and 
$\phi_{k}$ obey the transformation rule 
\begin{equation}
\psi_{k}(x) = \sum_{j=0}^{k} \binom{k+j}{2j} \frac{2k+1}{2j+1} (ab)^{k-j}
\phi_{j}(x). 
\end{equation}
\noindent
The expression for $H_{n}$ in terms of $\phi_{k}$ 
follows directly from this. Moreover, 
$H_{n}^{2}(x) = g_{n}\left( (ax-bx^{-1})^{2} \right)$ and the result is 
obtained by applying the 
Cauchy-Schl\"omilch formula to $f(g_{n})$.
\end{proof}

\begin{Example}
\label{beauty-1}
For $n \in \mathbb{N}$, 
\begin{equation}
\int_{0}^{\infty} \left[ (x^2+x^{-6})(x^4-x^2+1)-1 \right] 
e^{-(x^{7}-x^{-7})^{2n}} \, dx = \frac{1}{14n} \Gamma \left( \frac{1}{2n} 
\right).
\victor
\nonumber
\end{equation}
\noindent
In order to verify this value, write the Laurent polynomials of the 
integrand according to (\ref{form-d})
\begin{equation}
x^{7}-x^{-7} = 7y + 14y^{3} + 7y^{5} +y^{7}, 
\nonumber
\victor
\end{equation}
\noindent
and write the integrand as a function of $y = x - x^{-1}$, to obtain
\begin{equation}
7(x^{2}+x^{-6})(x^4-x^2+1)-7 = 7 + 42y^{2} + 35y^{4} + 7y^{6}.
\victor
\nonumber
\end{equation}
Implement (\ref{form-d1}) on 
\begin{equation}
f(y^{2}) = (7 + 42y^{2} + 35y^{4} + 7y^{6}) e^{-(y^{2}(7+14y^{2} + 7y^{4} + 
y^{6})^{2})^{n}}, 
\end{equation}
\noindent
to yield
\begin{equation}
\int_{0}^{\infty} 7 \left[ (x^2+x^{-6})(x^4-x^2+1)-1 \right] 
e^{-(x^{7}-x^{-7})^{2n}} \, dx = 
\int_{0}^{\infty} f(y^{2}) dy. 
\nonumber
\end{equation}
\noindent
The last step is an outcome of the 
substitution $z = 7y + 14y^{3} + 7y^{5} + y^{7}$. 
Hence
\begin{eqnarray}
\int_{0}^{\infty} f(y^{2}) dy & = &
\int_{0}^{\infty} (7 + 42y^{2} + 35y^{4} + 7y^{6}) 
e^{-(7y+14y^{3}+7y^{5}+y^{7})^{2n}} \, dy  \nonumber \\
& = & \int_{0}^{\infty} e^{-z^{2n}} \, dz = \frac{1}{2n} 
\Gamma \left( \frac{1}{2n} \right). \nonumber
\end{eqnarray}
\end{Example}

\begin{Example}
Proceeding  as in the previous example, we obtain
\begin{equation}
\int_{0}^{\infty} \left[ 7(x^{2}+x^{-6})(x^{4}-x^{2}+1) - 6 \right] 
e^{-(x+x^{7} -x^{-1}-x^{-7})^{2n}} \, dx = 
\frac{1}{2n} 
\Gamma \left( \frac{1}{2n} \right). \nonumber
\victor
\end{equation}
\noindent
The details are left to the reader. 
\end{Example}

\begin{Example}
The choice $y = x-x^{-1}, \, x^{3}-x^{-3} = 3y+y^{3}$ followed by 
$z= 3y+y^{3}$ produces
\begin{equation}
\int_{0}^{\infty} \frac{x^{4}-x^{2}+1 }{x^{2} } 
\prod_{j=0}^{\infty} \left[ 1 + (x^{3}-x^{-3})^{2} \, \nu^{2j} \right]^{-1}
\, dx   = 
\nonumber
\end{equation}
\begin{eqnarray}
\quad \quad   & =  & \int_{0}^{\infty}  (1 + y^{2}) 
\prod_{j=0}^{\infty} \left[ 1 + (3y+y^{3}))^{2} \, \nu^{2j} \right]^{-1}
\, dy \nonumber \\
\quad \quad & = &  \frac{1}{3} \int_{0}^{\infty} \prod_{j=0}^{\infty} 
(1+ z^{2} \nu^{2j})^{-1} \, dz \nonumber \\
 \quad \quad  & = & \frac{\pi}{6} (1 + \nu + \nu^{3} + \nu^{6} + \nu^{10} + \cdots )^{-1}. 
\nonumber
\end{eqnarray}
\end{Example}
\section{An integral due to Laplace} \label{sec-laplace} 
\setcounter{equation}{0}

The example described in this section is the original problem to which the 
Cauchy-Schl\"omilch transformation was applied. 

\begin{Example}
\label{example2.1}
The {\em normal integral} is 
\begin{equation}
\ift e^{-y^{2}} \, dy  =  \frac{\sqrt{\pi}}{2}. 
\label{normal-1}
\victor
\end{equation}
\noindent
The reader will find in \cite{irrbook} a variety of proofs of this fundamental
identity. 

Take $f(x) = e^{-x}$ in Theorem \ref{main-thm} to obtain
\begin{equation}
\ift e^{-(ax-b/x)^{2}} \; dx  =  \frac{\sqrt{\pi}}{2a}. 
\victor
\label{int-gr1}
\end{equation}
\no
Expanding the integrand and replacing the parameters $a$ and $b$ by their 
square roots produces entry $3.325$ in \cite{gr}: 
\begin{equation}
\int_{0}^{\infty} \text{exp} \left( -ax^{2}- b/x^{2} \right) \, dx = 
\frac{1}{2} \sqrt{\frac{\pi}{a}} e^{-2 \sqrt{ab}}.
\label{rel-1}
\victor
\end{equation}

The change of variable $x  = \sqrt{b}t/\sqrt{a}$ shows that the 
result (\ref{rel-1}) can be written in terms of a single parameter $c = ab$ as
\begin{equation}
\ift e^{-c(t -1/t)^{2}} \, dt  =  \frac{1}{2} \sqrt{\frac{\pi}{c}}.
\victor
\label{single-para}
\end{equation}
\end{Example}

\medskip

\begin{Example}
A host of other entries in \cite{gr} are amenable to  the 
Cauchy-Schl\"omilch 
transformation. For example, $3.324.2$ states that
\begin{equation}
\int_{-\infty}^{\infty} \text{exp} \left[ -(x - b/x)^{2n} \right] \, dx = 
\frac{1}{n} \Gamma \left( \frac{1}{2n} \right).
\label{rel-2}
\victor
\end{equation}
\noindent
This is now evaluated by choosing $f(x) = e^{-x^{n}}$ in
Theorem \ref{main-thm} so that
\begin{equation}
\int_{-\infty}^{\infty} \text{exp} \left[ -(x - b/x)^{2n} \right] \, dx   = 
 2 \int_{0}^{\infty} e^{-y^{2n}} \, dy. \nonumber 
\victor
\end{equation}
\noindent
The change of variables $t = y^{2n}$ and the
integral representation for the {\em gamma function}
\begin{equation}
\Gamma(a) = \int_{0}^{\infty} e^{-t}t^{a-1} \, dt
\victor
\end{equation}
\noindent
imply (3.5).
\end{Example}

\medskip

\begin{Example}
The expression $ t - 1/t$ in (\ref{single-para}) suggests
a natural change of variables $t = e^{u}$. This yields
\begin{equation}
\int_{-\infty}^{\infty} e^{u - c \text{ sinh}^{2}u } \, du  =  
\sqrt{\frac{\pi}{c}}. 
\victor
\label{sinh-1}
\end{equation}
\no
The latest version of \texttt{Mathematica} is unable to produce this result 
when $c$ is an arbitrary parameter. It does evaluate (\ref{sinh-1}) if $c$
is assigned a specific real value. 
\end{Example}

\section{An integral with three parameters} \label{sec-threeparam} 
\setcounter{equation}{0}

The introduction of parameters in a definite integral provides a greater 
flexibility in its evaluation. Many classical integrals are presented in 
\cite{bomo1} as special cases of the next theorem, which now appears 
as $3.242.2$ in \cite{gr}. The proof given below is in the spirit of the 
original observation of A. Cauchy. 

\begin{Thm}
\label{boros-3param}
Let 
\begin{eqnarray}
I_{1} &  = & \int_{0}^{\infty} \left( \frac{x^{2}}{x^{4}+2ax^{2}+1} \right)^{c} 
\cdot \frac{x^{2}+1}{x^{b}+1} \, \frac{dx}{x^{2}} \nonumber \\
I_{2} &  = & \int_{0}^{\infty} \left( \frac{x^{2}}{x^{4}+2ax^{2}+1} \right)^{c} 
\, \frac{dx}{x^{2}} \nonumber \\
I_{3} &  = & \int_{0}^{\infty} \left( \frac{x^{2}}{x^{4}+2ax^{2}+1} \right)^{c} 
\, dx \nonumber \\
I_{4} &  = & \frac{1}{2}
\int_{0}^{\infty} \left( \frac{x^{2}}{x^{4}+2ax^{2}+1} \right)^{c} 
\, \frac{x^{2}+1}{x^{2}} \, dx. \nonumber
\end{eqnarray}
\noindent
Then $I_{1}=I_{2}=I_{3}=I_{4}$ and this common value is 
\begin{equation}
I(a,b;c) = 2^{-1/2-c}(1+a)^{1/2-c} B \left( c - \tfrac{1}{2}, \tfrac{1}{2} 
\right). 
\label{master}
\end{equation}
\end{Thm}

\begin{proof}
Observe that if
$g$ satisfies $g(1/x) = x^{2}g(x)$, differentiation with 
respect to the parameter $b$ shows
that the integral of $g(x)/(x^{b}+1)$ 
over $[0, \infty)$ is independent of $b$.  This proves the equivalence 
of the four stated integrals. 

Theorem \ref{main-thm} is now used to evaluate $I_{3}$. For any function $f$, 
the Cauchy-Schl\"omilch transformation gives
\begin{eqnarray}
\int_{0}^{\infty} f \left( \frac{x^{2}}{x^{4}+2ax^{2}+1} \right) \, dx & = & 
\int_{0}^{\infty} f \left( \frac{1}{(x - x^{-1})^{2}+2a+2} \right) \, dx 
\nonumber \\
& = & \int_{0}^{\infty} f \left( \frac{1}{x^{2}+2(a+1)} \right). \nonumber
\end{eqnarray}
\noindent
Apply this to $f(x) = x^{c}$ and use the 
change of variables $u = \frac{2(a+1)}{x^{2}+2(a+1)}$ to produce
\begin{equation}
\int_{0}^{\infty} \left( \frac{x^{2}}{x^{4}+2ax^{2}+1} \right)^{c} \, dx  =  
\frac{1}{2} \left[ 2(a+1) \right]^{\tfrac{1}{2}-c} 
\int_{0}^{1} u^{c-3/2} (1-u)^{-1/2}. \nonumber
\end{equation}
\noindent
This last integral is the special value 
$B(c - \tfrac{1}{2}, \tfrac{1}{2})$ of Euler's beta function. \\
\end{proof}

The next theorem presents an alternative form of the integral in Theorem 
\ref{main-thm}. 

\begin{Thm}
\label{thm-alter}
For any function $f$, 
\begin{equation}
\int_{0}^{\infty} f \left( \frac{bx^{2}}{x^{4}+2ax^{2}+1} \right) \, dx =
\frac{\sqrt{b}}{2 \sqrt{a_{*}}} 
\int_{0}^{1} \frac{f(a_{*}t)}{\sqrt{t(1-t)}} \, \frac{dt}{t},
\end{equation}
\noindent
where $a_{*} = \frac{b}{2(1+a)}$. 
\end{Thm}
\begin{proof}
This follows from the identity in Theorem \ref{main-thm} and the 
change of variable $t = 2(a+1)/[ x^{2}+2(a+1)]$. 
\end{proof}

The {\em master formula} (\ref{master}) yields many other evaluations of 
definite integrals; see \cite{roopa} for some of them.  The next theorem
provides a new class of integrals that are derived from (\ref{master}). 

\begin{Thm}
\label{thm-series}
Suppose 
\begin{equation}
f(x) = \sum_{n=1}^{\infty} c_{n}x^{n}
\nonumber
\end{equation}
\noindent
be an analytic function with $f(0) = 0$. Then
\begin{equation}
\int_{0}^{\infty} f \left( \frac{x^{2}}{x^{4}+2ax^{2}+1} \right) \, dx = 
\frac{\pi}{2^{3/2} \, \sqrt{1+a}} \sum_{n=0}^{\infty} c_{n+1} \binom{2n}{n}
u^{n},
\label{series-for}
\end{equation}
\noindent
where $u = \frac{1}{8(1+a)}$. 
\end{Thm}
\begin{proof}
Integrate term-by-term and use the value
\begin{equation}
B \left( m + \tfrac{1}{2}, \tfrac{1}{2} \right) = 
\frac{\pi}{2^{2m}} \binom{2m}{m}
\nonumber
\victor
\end{equation}
\noindent
to simplify the result.
\end{proof}

\medskip

\section{Exponentials and Bessel functions} \label{sec-exponential} 
\setcounter{equation}{0}

This section describes the application of Theorem \ref{thm-series} to a 
number of definite integrals. The Taylor expansion of 
$f(x) = 1-e^{-bx}$ with $b > 0$ leads to  an integral that can be 
evaluated in terms of the 
{\em modified Bessel} functions $I_{\nu}(x)$  defined by 
the series
\begin{equation}
I_{\nu}(x) = \sum_{j=0}^{\infty} \frac{x^{\nu + 2j}
}{j! \Gamma(\nu + j+1) \, 2^{\nu+ 2j}}.
\victor
\end{equation}
\noindent
In particular 
\begin{equation}
I_{0}(x) = \sum_{j=0}^{\infty} \frac{x^{2j}}{2^{2j} j!^{2}} \text{ and }
I_{1}(x) = \sum_{j=0}^{\infty} \frac{x^{2j+1}}{2^{2j+1} j!(j+1)!}.
\label{bessel-exp}
\victor
\end{equation}
Although, as will be seen below, our results below have more direct derivations, the following procedure is more informative.

\begin{Example}
\label{bessel-1}
For $a>-1$ and $b>0$, let $c = \frac{b}{8(1+a)}$. Then 
\begin{equation}
\int_{0}^{\infty} \left( 1 - e^{-bx^{2}/(x^{4}+2ax^{2}+1)} \right) \, dx = 
\frac{\pi b e^{-2c} }{2^{3/2} \sqrt{1+a}}  
\left[ I_{0}(2c) + I_{1}(2c) \right].
\nonumber
\victor
\end{equation}

\medskip

\noindent
The function $f(x) = 1 - e^{-bx}$ has coefficients $c_{n} = (-1)^{n+1}b^{n}/n!$
and (\ref{series-for}) yields 
\begin{equation}
\int_{0}^{\infty} \left( 1 - e^{-bx^{2}/(x^{4}+2ax^{2}+1)} \right) \, dx = 
\frac{\pi b}{2^{3/2} \, \sqrt{1+a}}  h(- b u ) 
\nonumber
\end{equation}
\noindent
where $u = 1/8(1+a)$ and 
\begin{equation}
h(x) = \sum_{n=0}^{\infty} \frac{\binom{2n}{n} }{(n+1)!} x^{n}. 
\label{h-def}
\end{equation}
\end{Example}

\smallskip
The result now follows from the relation $c = bu$ and an identification of the 
series $h$ in terms of Bessel functions. 

\begin{Prop}
\label{bessel-lemma}
The following identity holds: 
\begin{equation}
\sum_{n=0}^{\infty} \frac{\binom{2n}{n}}{(n+1)!} x^{n} 
= e^{2x} \left[ I_{0}(2x) - I_{1}(2x) \right].
\nonumber
\victor
\end{equation}
\end{Prop}

We present two different proofs. The 
first one is elementary and is based on the WZ-method described 
in  \cite{aequalsb}. \texttt{Mathematica} actually provides a third proof 
by direct evaluation of the series. 

\begin{proof}
The expansion (\ref{bessel-exp}) yields
\begin{equation}
I_{0}(2x) - I_{1}(2x) = \sum_{r=0}^{\infty} \frac{x^{r}}{b_{r}},
\label{identity-0}
\end{equation}
\noindent
where 
\begin{equation}
b_{r} = 
\begin{cases} 
j!^{2} & \text{ if } r = 2j \\
-j!(j+1)! & \text{ if } r = 2j+1. 
\end{cases}
\nonumber
\end{equation}
\noindent
Multiplying (\ref{identity-0}) with the series for  $e^{2x}$ lends itself to 
an equivalnet formulation of the claim as the identity
\begin{equation}
\sum_{j=0}^{k} \frac{(-1)^{j}}{2^{j}} \binom{k}{j} 
\binom{j}{\lfloor{j/2\rfloor}} = 
\frac{1}{2^{k} (k+1)} \binom{2k}{k}.
\victor
\label{wz-1}
\end{equation}

The upper index of the sum is extended to infinity and (\ref{wz-1}) is 
written as 
\begin{equation}
\sum_{j \geq 0} \frac{(-1)^{j}}{2^{j}} \binom{k}{j} 
\binom{j}{\lfloor{j/2\rfloor}} = 
\frac{1}{2^{k} (k+1)} \binom{2k}{k}.
\nonumber
\victor
\end{equation}
\noindent
The even and odd indices are considered separately. Define
\begin{equation}
S_{e} := \sum_{j \geq 0} \frac{1}{2^{2j}} \binom{k}{2j} \binom{2j}{j} 
\text{ and } 
S_{o} := \sum_{j \geq 0} \frac{1}{2^{2j+1}} \binom{k}{2j+1} \binom{2j+1}{j}.
\victor
\nonumber
\end{equation}
\noindent
The result is obtained from the values
\begin{equation}
S_{e} = \frac{1}{2^{k}} \binom{2k}{k} \text{ and } 
S_{o} = \frac{k}{(k+1) \, 2^{k}} \binom{2k}{k}. 
\victor
\label{values}
\end{equation}
\noindent
To establish (\ref{values}), the WZ-method is applied to the functions
\begin{equation}
S_{e}^{*}(k) = S_{e} \binom{2n}{k}^{-1} 2^{k} \text{ and }
S_{o}^{*}(k) = S_{o} \binom{2k}{k}^{-1} \frac{k+1}{k}. 
\victor
\nonumber
\end{equation}
\noindent
The output is that both $S_{e}^{*}$ and $S_{o}^{*}$ satisfy the 
recurrence $a_{k+1} - a_{k} = 0$ with {\em certificates}  
\begin{equation}
\frac{-4j^{2}}{(2k+1)(k+1+2j)} \text{ and }
\frac{-4j(j+1)}{(2k+1)(k-2j)},
\nonumber
\end{equation}
respectively. The 
initial conditions $S_{e}^{*}(0) = S_{o}^{*}(0) = 1$ give
$S_{e}^{*}(k) \equiv S_{o}^{*}(k) \equiv  1$.
\end{proof}

\medskip

As promised above we present an alternative proof of \ref{bessel-lemma} based on Theorem \ref{thm-alter}.

\begin{proof}

Theorem \ref{thm-alter}  gives
\begin{equation}
I  :=  \int_{0}^{\infty} \left( 1 - e^{-bx^{2}/(x^{4}+2ax^{2}+1)} \right) \, dx
 =  \frac{\sqrt{b}}{2 \sqrt{a^{*}}} 
\int_{0}^{1} \frac{1-e^{-a^{*}t}}{t} \frac{dt}{\sqrt{t(1-t)}}. \nonumber 
\victor
\end{equation}
The latter is known as {\em Frullani integral} and  can be written as 
\begin{equation}
I = \frac{\sqrt{b a^{*}}}{2} 
 \int_{0}^{1} \int_{0}^{1} \frac{e^{-a^{*}ty}}{\sqrt{t(1-t)}} dy \, dt. 
\end{equation}
\noindent
Exchanging the order of integration, the inner integral is 
a well-known Laplace transform 
\cite{erderly3}(p. $366$, $19.5.11$ with $n=0$)
\begin{equation}
\int_{0}^{1} \frac{e^{-\omega t} \, dt}{\sqrt{t(1-t)}} = 
\pi e^{-\omega /2} I_{0} \left( \tfrac{\omega}{2} \right) 
\victor
\end{equation}
 whence we find
\begin{equation}
I = \frac{\pi \sqrt{b}}{\sqrt{a^{*}}} \int_{0}^{a^{*}/2} e^{-t} I_{0}(t) \, dt.
\nonumber
\end{equation}
\noindent
The relation
\begin{equation}
\frac{d}{dt} \left( te^{-t} ( I_{0}(t) + I_{1}(t) ) \right) = 
e^{-t} I_{0}(t)
\victor
\end{equation}
\noindent
now completes the proof.  
\end{proof}

\medskip

\begin{Example}
Choosing $a=0$ and $b=4$ gives 
\begin{equation}
\int_{0}^{\infty} \left( 1 - e^{-4x^{2}/(x^{4}+1)} \right) \, dx = 
\frac{\pi \sqrt{2}}{e} \left( I_{0}(1) + I_{1}(1) \right).
\nonumber
\victor
\end{equation}
\end{Example}

\begin{Example}
The values $a=1$ and $b=8$ yield
\begin{equation}
\int_{0}^{\infty} \left( 1 - e^{-8x^{2}/(x^{2}+1)^{2}} \right) \, dx = 
\frac{2 \pi}{e} \left( I_{0}(1) + I_{1}(1) \right).
\nonumber
\victor
\end{equation}
\end{Example}

\medskip

\noindent
\texttt{Mathematica} is unable to evaluate these two examples.  \\

\section{Trigonometric and Bessel functions} \label{sec-trigo} 
\setcounter{equation}{0}

The next example employs the familar Taylor 
expansion of $\sin b x$. The result is 
expressed in terms of the {\em Bessel function of the first kind}
\begin{equation}
J_{\nu}(x) = \sum_{k=0}^{\infty} \frac{(-1)^{k}}{k! \, (k+\nu)!} 
\left( \frac{x}{2} \right)^{2k + \nu}. 
\victor
\label{bes-def}
\end{equation}

\begin{Example}
\label{cool}
Let $c = b/8(1+a)$. Then
\begin{equation}
\int_{0}^{\infty} \sin \left( \frac{bx^{2}}{x^{4}+2ax^{2}+1} \right) \, dx =
\frac{\pi b }{\sqrt{8(1+a)}} 
\left[ J_{0}(2c) \cos 2c  + J_{1}(2c) \sin 2c \right].
\victor
\nonumber
\end{equation}
\noindent

\smallskip
To verify this, apply Theorem \ref{thm-series}  to $\sin bx$ to obtain
\begin{equation}
\int_{0}^{\infty} \sin \left( \frac{bx^{2}}{x^{4}+2ax^{2}+1} \right) \, dx =
\frac{\pi b }{\sqrt{8(1+a)}} 
\sum_{k=0}^{\infty} \frac{(-1)^{k}}{(2k+1)!} \binom{4k}{2k} c^{2k}.
\nonumber
\victor
\end{equation}

\begin{Lem}
The following identity holds:
\begin{equation}
\sum_{k=0}^{\infty} \frac{(-1)^{k}}{(2k+1)!} \binom{4k}{2k} c^{2k} = 
J_{0}(2c) \cos 2c + J_{1}(2c) \sin 2c.
\victor
\nonumber
\end{equation}
\noindent

\end{Lem}
\begin{proof}
Using the Cauchy product and the series expression (\ref{bes-def}), the 
right-hand side is written as
\begin{equation}
J_{0}(2c) \cos 2c + J_{1}(2c) \sin 2c  = \nonumber
\end{equation}
\begin{eqnarray}
& = & \sum_{k=0}^{\infty} \sum_{j=0}^{k} \frac{(-1)^{k} c^{2k} 4^{j}}
{(2j)! (k-j)!^{2}} + 
2c^{2} \sum_{k=0}^{\infty} \sum_{j=0}^{k} \frac{(-1)^{k} c^{2k} 4^{j}}
{(2j+1)! (k-j)! (k-j+1)!} \nonumber \\
& = & \sum_{k=0}^{\infty} \frac{(-1)^{k} c^{2k} (4k)!}{(2k)!^{3}} + 
2c^{2} \sum_{k=0}^{\infty} \frac{(-1)^{k} c^{2k} (4k+3)!}{(2k+3)! (2k+1)^{2}} 
\nonumber \\
& = & \sum_{k=0}^{\infty} \frac{(-1)^{k}}{(2k+1)!} \binom{4k}{2k} c^{2k}.
\nonumber
\end{eqnarray}
\noindent
The passage from the first to the second equality is justified by the 
identities 
\begin{equation}
\sum_{j=0}^{k} \frac{4^{j}}{(2j)! (k-j)!^{2}} = 
\frac{1}{(2k)!} \binom{4k}{2k} \text{ and }
\sum_{j=0}^{k} \frac{4^{j}}{(2j)! (k-j)! \, (k-j+1)!} = 
\frac{4k+3}{(2k+3)!} \binom{4k+2}{2k+1}. \nonumber
\end{equation}
\noindent
Both of these formulas are in turn verifiable via the WZ-method \cite{aequalsb}
with their respective rational certificates
\begin{equation}
\frac{(6n^{2}+10n+4-4nk-34k)(2k-1)k}{(n+1-k)^{2} (4n+1)(4n+3)} \text{ and }
\frac{(20n+17+6n^{2}-4nk-7k)(2k-1)}{(n+1-k)(n+2-k) (4n+5)(4n+7)}.
\nonumber
\end{equation}
\end{proof}

\medskip

\begin{Example}
The choice $a=0$ and $b=1$ in Example \ref{cool} produces 
\begin{equation}
\int_{0}^{\infty} \sin \left( \frac{x^{2}}{x^{4}+1} \right) \, dx = 
\frac{\pi}{2 \sqrt{2}} \left[
J_{0} \left( \tfrac{1}{4} \right) 
\cos \left( \tfrac{1}{4} \right) + 
J_{1} \left( \tfrac{1}{4} \right) 
\sin \left( \tfrac{1}{4} \right)
\right].
\label{nice-trick}
\victor
\end{equation}
\end{Example}

\begin{Example}
By choosing $a=b=1$ in Example \ref{cool}, we get 
\begin{equation}
\int_{0}^{\infty} 
\sin \left( \left[\frac{x}{x^{2}+1} \right]^{2} \right) \, dx = 
\frac{\pi}{4} \left[
J_{0} \left( \tfrac{1}{8} \right) 
\cos \left( \tfrac{1}{8} \right) + 
J_{1} \left( \tfrac{1}{8} \right) 
\sin \left( \tfrac{1}{8} \right)
\right]. 
\victor
\end{equation}
\end{Example}

\noindent
As the time of this writing, \texttt{Mathematica} is unable to evaluate 
the integrals in the two previous examples. \\

\noindent
{\bf Note}. The function 
\begin{equation}
g(u) = \sum_{k=0}^{\infty} \frac{(-1)^{k}}{(2k+1)!} \binom{4k}{2k} 
u^{2k}
\end{equation}
\noindent
also has a hypergeometric form  as 
\begin{equation}
g(u) = 
{}_{2}F_{3}\bigg[{\tfrac{1}{4} \,\,\, \tfrac{3}{4} \atop \tfrac{1}{2} \,\,\, 
 1 \,\,\, 
\tfrac{3}{2}}; -4u^{2} \bigg].
\label{standard}
\victor
\end{equation}
\noindent
This follows directly from the identity
\begin{equation}
\frac{\binom{4k}{2k}}{(2k+1)!} = 
\frac{2^{2k-3/2} \Gamma(k + 1/4) \Gamma(k + 3/4)}
{\Gamma(k+1/2) \, \Gamma^{2}(k+1) \, \Gamma(k + 3/2)}, 
\victor
\nonumber
\end{equation}
\noindent
that is established via the duplication formula of the gamma function
\begin{equation}
\Gamma(2x) = \frac{1}{\sqrt{\pi}} 2^{2x-1} \Gamma(x) \Gamma(x + \tfrac{1}{2})
\nonumber
\victor
\end{equation}
\noindent
and its iteration
\begin{equation}
\Gamma(4x) = \frac{1}{\pi \sqrt{\pi}} 2^{8x-5/2} 
\Gamma(x) \Gamma(x + \tfrac{1}{4})
\Gamma(x + \tfrac{1}{2} ) \Gamma(x + \tfrac{3}{4})
\nonumber
\victor
\end{equation}
\noindent
Then $\Gamma(a+k) = (a)_{k} \Gamma(a)$ produces (\ref{standard}).  Here 
$(a)_{k} = a(a+1) \cdots (a+k-1)$ is the Pochhammer symbol.  This gives 
an alternative form of the result described in Example \ref{cool}, i.e., 

\begin{equation}
\int_{0}^{\infty} \sin \left( \frac{bx^{2}}{x^{4}+2ax^{2}+1} \right) \, dx =
\frac{\pi \, b}{2 \sqrt{2(1+a)}} \,  \, 
{}_{2}F_{3}\bigg[{\tfrac{1}{4} \,\,\, \tfrac{3}{4} \atop \tfrac{1}{2} \,\,\, 
 1 \,\,\, 
\tfrac{3}{2}}; -b^{2}/16(1+a)^{2} \bigg].
\victor
\nonumber
\end{equation}

\medskip

\noindent
{\bf Note}. Proceeding as in Example \ref{beauty-1} we can obtain
\begin{multline}
\int_{0}^{\infty} (x^{2} + x^{-2}-1) 
\sin \left( \frac{x^{6}+x^{-6}-2}{x^{12} - 4x^{6} - 4x^{-6} + x^{-12} +7} 
\right) \, dx =  \\
\frac{\pi}{6 \sqrt{2}} \left[ 
J_{0} \left( \tfrac{1}{4} \right)  \cos \tfrac{1}{4}+ 
J_{1} \left( \tfrac{1}{4} \right)  \sin \tfrac{1}{4}
\right]. 
\nonumber
\victor
\end{multline}
\noindent
\texttt{Mathematica} is unable to compute the preceding integral.  \\

\noindent
{\bf Note}. The result in Example \ref{cool} 
can also be established by the Frullani
method described in Section \ref{sec-exponential}. Start with
\begin{equation}
I := \int_{0}^{\infty} \sin \left( \frac{bx^{2}}{x^{4}+2ax^{2}+1} \right) 
\, dx = \frac{\sqrt{b}}{2 \sqrt{a^{*}}} \int_{0}^{1} \frac{\sin(a^{*}t)}{t} 
\frac{dt}{\sqrt{t(1-t)}}. 
\nonumber
\end{equation}
\noindent
As before, write
\begin{equation}
I = \frac{\sqrt{a^{*} b}}{2} \int_{0}^{1} \int_{0}^{1} \frac{\cos(a^{*}ty)}
{\sqrt{t(1-t)}} dt \, dy. 
\nonumber
\end{equation}
\noindent
The inner integral is a well-known cosine transform 
(\cite{erderly2},p.12, $1.3.14$\footnote{Note that formula $1.3.14$ in 
\cite{erderly2} is incorrect.} 
\begin{equation}
\int_{0}^{1} \frac{\cos( \omega t) \, dt}{\sqrt{t(1-t)}} = 
\pi J_{0} \left( \frac{\omega}{2} \right) \cos \frac{\omega}{2}, 
\victor
\end{equation}
\noindent
and it follows that
\begin{equation}
I = \frac{\pi \sqrt{b}}{\sqrt{a^{*}}} \int_{0}^{a^{*}/2} 
\cos t \, J_{0}(t) \, dt. 
\nonumber
\end{equation}
\noindent
The result is now immediate from the identity
\begin{equation}
\frac{d}{dt} 
\left[ t \left( \cos t \, J_{0}(t) + \sin t \, J_{1}(t) \right) \right] = 
\cos t \, J_{0}(t).
\victor
\nonumber
\end{equation}
\end{Example}

\section{The sine integral and Bessel functions}
\label{sec-sine}
\setcounter{equation}{0}

The {\em sine integral} is defined by 
\begin{equation}
\text{Si}(x) := \int_{0}^{x} \frac{\sin t}{t} \, dt. 
\end{equation}
\noindent 
The Cauchy-Schl\"omilch transformation can be employed to prove 
\begin{equation}
\int_{0}^{\infty} \text{Si} \left( \frac{bx^{2}}{x^{4}+2ax^{2}+1} \right) \, dx 
=  \pi \sqrt{2(1+a)} \, \left[ (4c \cos 2c - \sin 2c ) J_{0}(2c) + 
4c \sin 2c \, J_{1}(2c) \right],
\victor
\nonumber
\end{equation}
\noindent
where $c = b/8(1+a)$. To 
establish this identity, start with the evaluation in 
Example \ref{cool}
\begin{equation}
\int_{0}^{\infty} \sin \left( \frac{bx^{2}}{x^{4}+2ax^{2}+1} \right) \, dx =
\frac{\pi b }{\sqrt{8(1+a)}} 
\left[ J_{0}(2c) \cos 2c  + J_{1}(2c) \sin 2c \right]
\nonumber
\end{equation}
\noindent
divide by $b$ and integrate both sides. Then  the identity
\begin{equation}
\frac{d}{dx} \left[ (2x \cos x - \sin x ) J_{0}(x)  + 2x \sin x J_{1}(x) 
\right] = J_{0}(x) \cos x + J_{1}(x) \sin x 
\victor
\end{equation}
\noindent
gives the result. 

\begin{Example}
The special case $a=0$ and $b=1$ yields 
\begin{equation}
\int_{0}^{\infty} \text{Si} \left( \frac{x^{2}}{x^{4}+1} \right) \, dx = 
\frac{\pi}{2 \sqrt{2}} \left[ J_{0}( \tfrac{1}{4})  \left[ \cos \tfrac{1}{4} 
- 2 \sin \tfrac{1}{4} \right] +
J_{1}( \tfrac{1}{4})  \sin \tfrac{1}{4}  \right]. 
\nonumber
\victor
\end{equation}
\end{Example}

\begin{Example}
The special case $a=b=1$ yields 
\begin{equation}
\int_{0}^{\infty} \text{Si} 
\left( \left[ \frac{x}{x^{2}+1} \right]^{2} \right) \, dx = 
\frac{\pi}{2} \left[ J_{0}( \tfrac{1}{8})  \left[ \cos \tfrac{1}{8} 
- 4 \sin \tfrac{1}{8} \right] 
 + J_{1}( \tfrac{1}{8})  \sin \tfrac{1}{8}  \right]. 
\nonumber
\victor
\end{equation}

\medskip

\noindent
\texttt{Mathematica} is unable to evaluate these integrals. 
\end{Example}
\medskip

\section{The Riemann zeta function} \label{sec-zeta} 
\setcounter{equation}{0}

Interesting examples of definite integrals 
come from integral representations of special functions. 
For the Riemann zeta function 
\begin{equation}
\zeta(s) = \sum_{k=1}^{\infty} \frac{1}{n^{s}},
\end{equation}
\noindent
one such expression is given by
\begin{equation}
\zeta(s)  =  \frac{1}{(1 - 2^{1-s}) \, \Gamma(s) } 
\ift \frac{t^{s-1} \, dt}{1+e^{t}}. 
\label{zeta-int}
\victor
\end{equation}
\noindent 
Analytic properties of $\zeta(s)$ are often established via such
integral formulas.

The change of variables $t = y^{2}$ produces
\begin{equation}
\int_{0}^{\infty} \frac{y^{2s-1} \, dy}{1+e^{y^{2}}} = 
\frac{1}{2} (1-2^{1-s})\Gamma(s) \zeta(s).
\label{zeta-def1}
\victor
\end{equation}
\noindent
\begin{Example}
We now employ Theorem \ref{main-thm} to establish
\begin{equation}
\int_{0}^{\infty} \frac{x^{2s+1} \, dx}{\cosh^{2}(x^{2})} = 
2^{-s} (1-2^{1-s}) \Gamma(s+1) \zeta(s).
\victor
\label{ex-zeta-2}
\end{equation}
The notation 
\begin{equation}
\Lambda(s) 
 := \frac{1-2^{1-s}}{2} \Gamma(s) \zeta(s)
\end{equation}
\noindent
is  employed in the proof. First
introduce the change of variable $x = t^{r}$ in the
Cauchy- Schl\"omilch formula and take 
$a = b$ for simplicity.  Then
\begin{equation}
\ift t^{r -1} f \left[ a^{2}( t^{r} - t^{-r} )^{2} \right] \, dt 
 =  \frac{1}{ar} \ift f(y^{2}) \, dy. 
\label{sch-90}
\end{equation}
\noindent
Now let $f(x) = x^{s-1/2}/(1+e^{x})$. Using
the notation $S_{r} = \sinh(r \ln t)$, (\ref{sch-90}) yields
\begin{equation}
\frac{\Lambda(s)}{ar} = 
\int_{0}^{\infty} \frac{t^{r-1} (at-at^{-r})^{2s-1} \, dt }
{ 1 + \text{exp}[ (at^{r}-at^{-r})^{2} ]} = 
\int_{0}^{\infty} \frac{t^{r-1} (2aS_{r})^{2s-1} \, dt}
{1+ \text{exp}[ (2a S_{r})^{2} ]}. 
\label{form-2}
\end{equation}
\noindent
Differentiate (\ref{form-2}) with respect to $a$ and use
\begin{equation}
\frac{c}{( 1+ c)^{2}}  =  \frac{1}{4 \cosh^{2}(\theta/2)}
\label{cosh-11}
\victor
\end{equation}
to obtain
\begin{equation}
\frac{2s-1}{a} \frac{\Lambda(s)}{ar} - 2(2a)^{2s-1} 
\int_{0}^{\infty} \frac{t^{r-1} S_{r}^{2s+1} \, dt}{\cosh^{2}[ 2(aS_{r})^{2}]} 
= - \frac{\Lambda(s)}{a^{2}r},
\end{equation}
\noindent
that produces
\begin{equation}
\int_{0}^{\infty} \frac{t^{r-1} S_{r}^{2s+1} \, dt}{\cosh^{2}[ 2(aS_{r})^{2}] }
= \frac{2s \Lambda(s)}{(2a)^{2s} a^{2} r}. 
\label{int-two}
\end{equation}
\noindent
Now change $t$ by $t^{-1}$ in (\ref{int-two}) and average the resulting 
integral with itself. The outcome is written as 
\begin{equation}
\int_{0}^{\infty} \frac{t^{r-1} S_{r}^{2s+1} \, dt}{\cosh^{2}[ 2 (a S_{r})^{2}
] } = \frac{2s \Lambda(s)}{(2a)^{2s} a^{2} r}. 
\end{equation}
\noindent
The final change of variables $x = \sqrt{2} a S_{r}$ produces the stated 
result. 

The special case $s= \tfrac{1}{2}$ yields
\begin{equation}
\int_{0}^{\infty} \frac{x^{2} \, dx}{\cosh^{2}(x^{2})}  = 
-\frac{1}{4} ( 2 - \sqrt{2}) \zeta(1/2) \sqrt{\pi}. 
\label{ex-zeta-1}
\victor
\end{equation}
\noindent
\texttt{Mathematica} is unable to produce (\ref{ex-zeta-1}). 
\end{Example}

\section{The error function}
\label{sec-error}
\setcounter{equation}{0}

Several entries in  the table \cite{gr} involve the {\em error function}
\begin{equation}
\text{erf}(x) := \frac{2}{\sqrt{\pi}} \int_{0}^{x} e^{-t^{2}} \, dt. 
\victor
\end{equation}
\noindent
For instance, entry $3.466.1$ states that 
\begin{equation}
\int_{0}^{\infty} \frac{e^{-\mu^{2}x^{2}} \, dx}{x^{2}+ \beta^{2}} = 
\frac{\pi}{2 \beta} \left( 1 - \text{erf}(\mu \beta) \right) 
e^{\mu^{2} \beta^{2}}. 
\victor
\label{error-1}
\end{equation}
\noindent
\texttt{Mathematica} is able to compute this example, which can be 
checked by writing the exponential in the integrand as 
$e^{\mu^{2} \beta^{2}} \times e^{-\mu^{2}(x^{2}+\beta^2)}$ and 
differentiating with respect to $\mu^{2}$. 

\begin{Example}
The Cauchy-Schl\"omilch transformation is now applied to the 
function 
\begin{equation}
f(x) = e^{-\mu^{2}x^{2}}/(x^{2} + 2(a+1))
\end{equation}
\noindent
to produce
\begin{equation}
\int_{0}^{\infty} \frac{e^{-\mu^{2}(x^{2}+x^{-2})} \, dx }
{x^{2}+2a + x^{-2}}  = 
\frac{\pi e^{2a \mu^{2}}}{2 \sqrt{2(a+1)}} 
\left[ 1 - \text{erf} \left( \mu \sqrt{2(a+1)} \, \, \right) \right].
\victor
\end{equation}
\noindent
The choice $a=\mu=1$ yields
\begin{equation}
\int_{0}^{\infty} \frac{e^{-(x^2+x^{-2})} \, dx }{(x+x^{-1})^{2}} = 
\frac{\pi e^{2}}{4}
\left[ 1 - \text{erf} (2)\right],
\victor
\end{equation}
\noindent
and $a=0, \, \mu=1$ gives
\begin{equation}
\int_{0}^{\infty} \frac{e^{-(x^{2}+x^{-2})}\, dx }{x^{2}+x^{-2}} = 
\frac{\pi}{2 \sqrt{2}} \left[ 1 - \text{erf}(\sqrt{2}) \right]. 
\victor
\end{equation}
\noindent
Neither of these special cases 
is computable by the current version of \texttt{Mathematica}.
\end{Example}

\section{Elliptic integrals}
\label{sec-elliptic}
\setcounter{equation}{0}

The classical {\em elliptic integral of the first kind} is defined by
\begin{equation}
\mathbf{K}(k) := \int_{0}^{1} \frac{dx}{\sqrt{(1-x^{2})(1-k^{2}x^{2})}} =
\int_{0}^{\pi/2} \frac{d \varphi}{\sqrt{1 - k^{2} \sin^{2} \varphi}}.
\end{equation}
\noindent
The table \cite{gr} contains a variety of definite integrals that can be 
evaluated in terms of $\mathbf{K}(k)$. For instance, entry $3.843.4$ states 
that
\begin{equation}
\int_{0}^{\infty} \frac{\tan x}{\sqrt{1- k^{2} \sin^{2}(2x)}} \, 
\frac{dx}{x} = 
\mathbf{K}(k). 
\end{equation}
\noindent
The reader will find in \cite{lawden1} a large variety of examples. 

In the context of the 
Cauchy-Schl\"omilch transformation, we present two illustrative
examples. 

\begin{Example}
The first result is 
\begin{equation}
\int_{0}^{\infty} \frac{x^{2} \, dx}
{\sqrt{(x^{4}+2ax^{2}+1)(x^{4}+2bx^{2}+1)}} = 
\frac{1}{\sqrt{2(a+1)}} \mathbf{K} \left( \sqrt{ \frac{a-b}{a+1}} \right). 
\label{elliptic-1}
\victor
\end{equation}

\smallskip

To verify this result, apply (\ref{main-thm}) to the function 
\begin{equation}
f(x) = \frac{1}{\sqrt{(x+2a+2)(x+2b+2)}}
\nonumber
\end{equation}
\noindent
and observe that 
\begin{equation}
f \left( ( x - 1/x)^{2} \right) = 
\frac{x^{2}}{\sqrt{(x^{4}+2ax^{2}+1)(x^{4}+2bx^{2}+1)}}.
\nonumber
\end{equation}
\noindent
The Cauchy-Schl\"omilch transformation gives 
\begin{equation}
\int_{0}^{\infty} \frac{x^{2} \, dx}
{\sqrt{(x^{4}+2ax^{2}+1)(x^{4}+2bx^{2}+1)}} = 
\int_{0}^{\infty} \frac{dx}{\sqrt{(x^{2}+A^{2})(x^{2}+B^{2})}}
\victor
\nonumber
\end{equation}
\noindent
with $A^{2} = 2(a+1)$ and $B^{2}=2(b+1)$. This last integral is 
computed by the change of variable $x = A \tan 
\varphi$ and the trigonometric form of the elliptic integral yields 
(\ref{elliptic-1}). 
\end{Example}

\begin{Example}
A similar argument produces the second elliptic integral evaluation. This 
time it involves
\begin{equation}
F( \varphi,k) := \int_{0}^{\varphi} \frac{dt}{\sqrt{1-k^{2} \sin^{2}t}} =
\int_{0}^{\sin \varphi} \frac{dx}{\sqrt{(1-x^{2})(1-k^{2}x^{2})}}
\end{equation}
\noindent
the (incomplete) elliptic integral of the first kind. 

\smallskip

Assume $a \leq b \leq c$. Then 
\begin{multline}
\int_{0}^{\infty} \frac{x^{3} \, dx}
{\sqrt{(x^{4}+2ax^{2}+1)(x^{4}+2bx^{2}+1)(x^{4}+2cx^{2}+1)}} =  \\
\frac{1}{2 \sqrt{(b+1)(c-a)}} 
F \left[ \, \sin^{-1} \sqrt{ \frac{c-a}{c+1}}, \, 
\sqrt{ \frac{ (b-a)(c+1)}{b+1)(c-a)}} \, \right]. 
\end{multline}
\end{Example}

\medskip

Following \cite{glasser4} a 
generalization of the Cauchy-Schl\"omilch identity 
is now used to evaluate some hyper-elliptic integrals. 

\begin{Thm}
\label{larry-gen}
Assume $\phi(z)$ is a meromorphic function with only
real simple poles $a_{j}$ with $\text{Res}(\phi; a_{j}) < 0$. Moroever 
assume $\phi$ is asymptotically linear. Then, for any even real valued 
function $f$, 
\begin{equation}
\int_{0}^{\infty} f \left[ \phi(x) \right] \, dx = 
\int_{0}^{\infty} f(x) \, dx. 
\end{equation}
\end{Thm}

\medskip

\begin{Example}
As a simple illustration we take 
\begin{equation}
\phi_{N}(z) = z \prod_{j=1}^{N} \frac{z^{2}- b_{j}^{2}}{z^{2} - a_{j}^{2}},
\end{equation}
\noindent
where $a_{1} < a_{2} < \cdots < a_{N}$. Take $N=1$ and write $b_{1}=b$. 
Theorem \ref{larry-gen}  and 
\begin{equation}
\int_{0}^{\infty} \frac{dx}{\sqrt{(x^{2} + \alpha^{2})(x^{2}+\beta^{2})}} =
\frac{1}{\alpha} \mathbf{K} \left( \frac{\sqrt{\alpha^{2}- \beta^{2}}}
{\alpha} \right), 
\end{equation}
\noindent
give
\begin{equation}
\int_{0}^{\infty} \frac{(t-b^{2})^{2} \, dt}
{\sqrt{t P(t) Q(t)}} = 
\frac{2}{\alpha} \mathbf{K} \left( \frac{\sqrt{\alpha^{2}- \beta^{2}}} 
{ \alpha} \right),
\end{equation}
\noindent
with 
\begin{eqnarray}
P(t) & = & t^{3} + ( \alpha^{2} - 2a^{2})t^{2} + ( a^{4} - 2 \alpha^{2}b^{2})t
+ \alpha^{2} b^{4}, \nonumber \\
Q(t) & = & t^{3} + ( \beta^{2} - 2a^{2})t^{2} + ( a^{4} - 2 \beta^{2}b^{2})t
+ \beta^{2} b^{4}. \nonumber
\end{eqnarray}

As an interesting special case, we have 
\begin{equation}
\int_{0}^{\infty} \frac{ (t - \tfrac{1}{2})^{2} \, dt}
{\sqrt{t \left[ t^{3} - 2 (kk')^{2} t + k^{2} \right] 
\left[ t^{3} - 4(kk')^{2} t^{2} + k^{4} \right]} }  = 
\frac{1}{k} \mathbf{K'}(k),
\end{equation}
\noindent
where $k'$ is the complementary modulus and $\mathbf{K'}(k) := 
\mathbf{K}(k')$. \texttt{Mathematica} is unable to 
deal with this. 
\end{Example}

\section{An extension of the 
Cauchy-Schl\"omilch method} \label{sec-extension} 
\setcounter{equation}{0}

An extension of Theorem \ref{main-thm}  by 
Jones \cite{jones1} is discussed here. The next section 
presents statistical applications of this result. 

\begin{Thm}
Let $s$ be a continuous decreasing function from ${\mathbb{R}}^{+}$ onto 
${\mathbb{R}}^{+}$. Assume $f$ is self-inverse, that is, $s^{-1}(x) = 
s(x)$ for all $x \in {\mathbb{R}}^{+}$. Then
\begin{equation}
\int_{0}^{\infty} f\left( [x - s(x) ]^{2} \right) \, dx = 
\int_{0}^{\infty} f\left( y^{2} \right) \, dy,
\label{gen-sch}
\end{equation}
\noindent
provided the integrals are convergent. 
\end{Thm}
\begin{proof}
The change of variables $t = s(x)$ yields
\begin{equation}
I = \int_{0}^{\infty} f( [x- s(x) ]^{2}) \, dx = 
- \int_{0}^{\infty} f( [s(t) - t ]^{2}) \, s'(t) \, dt. 
\end{equation}
\noindent
The average of these two representations, followed by the change of variables 
$u = x - s(x)$ gives the result. 
\end{proof}

\noindent
{\bf Note}. The above result is given without a scaling constant $a>0$. This
could be introduced by the change of variable $x_{1} = ax$ in 
(\ref{gen-sch}) to obtain, after relabeling $x_{1}$ as $x$, 
\begin{equation}
\int_{0}^{\infty} f( [ax- s(ax) ]^{2}) \, dx = 
\frac{1}{a} \int_{0}^{\infty} f\left( y^{2} \right) \, dy.
\label{parameter-1}
\end{equation}

Jones \cite{jones1} lists several specific forms of self-inverse $s(x)$ along 
with two methods for generating such functions based on work of Kucerovsky, 
Marchand and Small \cite{kucerovsky1}. 

\begin{Example}
An attractive self-inverse function is 
\begin{equation}
s(x) = x - \frac{1}{\alpha} \log \left( e^{\alpha x} -1 \right). 
\label{formula-fors}
\end{equation}
\noindent
Then (\ref{gen-sch}) becomes 
\begin{equation}
\int_{0}^{\infty} f \left( \frac{1}{\alpha^{2}} \log^{2}(e^{\alpha x} -1 ) 
\right) \, dx = \int_{0}^{\infty} f(y^{2}) \, dy. 
\end{equation}
\noindent
The choice $f(x) = e^{-x}$ gives, using (\ref{normal-1}), 
\begin{equation}
\int_{0}^{\infty} \text{exp}\left( - \frac{1}{\alpha^{2}} \log^{2}
(e^{\alpha x} -1) \right) \, dx = \frac{\sqrt{\pi}}{2}.
\label{jones-11}
\end{equation}
\end{Example}

\begin{Example}
Several other examples of self-inverse functions are provided in Jones
\cite{jones1}. Each one produces a Cauchy-Schl\"omilch 
type integral. Some examples are
\begin{eqnarray}
\int_{1}^{\infty} f \left[ (x - \text{exp}(\alpha/\log x))^{2} \right] \, dx 
& = & \int_{0}^{\infty} f(y^{2}) \, dy, \nonumber \\
& & \nonumber \\
\int_{0}^{\infty} f \left[ \frac{1}{\alpha^{2}} 
\, \log^{2} \left( \frac{e^{\alpha x} \sinh(\alpha x)}{1 + \cosh( \alpha x)} 
\right) \right] \, dx 
& = & \int_{0}^{\infty} f(y^{2}) \, dy, \nonumber \\
& & \nonumber \\
\int_{0}^{\infty} f \left[ (x - \sinh(\alpha/\sinh^{-1}x) )^{2} \right] \,
\, dx 
& = & \int_{0}^{\infty} f(y^{2}) \, dy. \nonumber 
\end{eqnarray}
\end{Example}

\section{Application to generating flexible probability distributions} 
\label{sec-applications} 
\setcounter{equation}{0}

There has recently been renewed interest in the statistical literature in 
generating flexible families of probability distributions for univariate 
continuous random variables.  Baker \cite{baker-rose} describes the use of 
Cauchy-Schl\"omilch transformation to generate new 
probability density functions from 
old ones. Jones \cite{jones1} does the same with the extended transformation 
of Section \ref{sec-extension}. The 
identity (\ref{parameter-1}) states that the total mass of 
$f(y^{2})$ is the same as that of $a f( [ ax - s(ax) ]^{2})$ for any 
self-inverse function $s$ and any scaling constant $a >0$. 

There are many techniques for introducing one or more parameters 
into a simple `parent distribution' with probability density function $g$, 
to produce more sophisticated distributions.  One such method, not
generally familiar, proceeds by `transformation of scale', defining
$f_b(x) \propto g(t_b(x))$ where $t_b(x)$ depends on the new 
parameter $b$. The difficulty associated with this procedure is the 
validation that $f_{b}$ is integrable and then to explicitly provide
its normalizing constant. 

The Cauchy-Schl\"omilch result in Theorem \ref{main-thm} guarantees that 
the choice $t_b(x) = |x-bx^{-1}|$ produces from 
the density $g$ of a positive random variable a new density $f_{b}$, also 
for a positive random variable, via 
\begin{equation}
f_b(x) = g(|x-bx^{-1}|). 
\label{12.1}
\end{equation}
\noindent
This was observed by Baker \cite{baker-rose}. The parameter $a$ 
in (\ref{transf1}) is redundant 
for distribution theory work since its action 
as a scale parameter is well understood; $a$ must, however, be 
reintroduced for practical fitting of such distributions to data.

A number of general properties of distributions with density of the form 
$f_b$ follow, some of which are:
\smallskip

\noindent (i) $f_b$ is R-symmetric \cite{mudholkar1} about R-center $\sqrt{b}$, 
i.e.\ $f_b(\sqrt{b}x) = f_b(\sqrt{b}x^{-1})$;
\smallskip

\noindent (ii) $f_b(0)=0$ and 
$$ f_b(x) \approx g(b/x) ~{\rm as}~x \rightarrow 0;~~~~~
f_b(x) \approx g(x) ~{\rm as}~x \rightarrow \infty ;$$

\smallskip
\noindent (iii) moment relationships follow  from
$$E_{f_b} \{ |X-bX^{-1})|^r\} = E_g(Y^r)$$
and, by R-symmetry,
$$E_{f_{b}}(X^r) = b^{r+1} E_{f_b}(X^{-(r+2)});$$ 

\smallskip
\noindent (iv) if $g$ is decreasing, then $f_b$ is unimodal with mode at 
$\sqrt{b}$; 

\smallskip
\noindent (v) if $f_b$ is unimodal, its mean and its median are both 
greater than its mode. \\

\noindent 
These properties can be 
found in Baker \cite{baker-rose}, but only special cases of (iii) are 
provided. 

Amongst the most interesting distributions presented by Baker
is the root-reciprocal inverse Gaussian distribution (RRIG). This example,  
also discussed in Mudholkar and Wang \cite{mudholkar1}, in the 
case of dispersion parameter 
$\lambda=1$, arises from (\ref{12.1}) when $g$ is the half-Gaussian density.  
The RRIG density is 
\begin{equation}
f_b(x) = \sqrt{\frac{2}{\pi}} e^b \exp\left\{ -\frac12 \left(
x^2+{b^2/ x^2 } \right) \right\}, \label{int-11}
\end{equation}
\noindent
$b >0$. This corresponds in integral terms to (\ref{rel-1}) above. Similarly, 
a second example presented by Baker \cite{baker-rose} (Section 3.4)
is the distribution based on the half-Subbotin 
distribution. This is directly linked to (\ref{rel-2}). A 
third example, based on the half-$t$ distribution, has 
density
\begin{equation}
f_{\nu,b}(x) = \frac{2 \Gamma((\nu+1)/2)}{\sqrt{\nu\pi}\, \Gamma(\nu/2)}
(1+ (x-b/x)^2/\nu)^{-(\nu+1)/2},
\label{fnub}
\end{equation}
\noindent
$\nu,b>0$.
The verification that $f_{\nu,b}(x)$ integrates to $1$ can be done by using 
the integral $I_3$ in Theorem \ref{boros-3param}. Of course, other 
distributions in Baker \cite{baker-rose} correspond to other 
integral formulae, while integral formulae in this article which have 
nonnegative integrands correspond to other distributions. 

Transformation of scale densities are particularly amenable to having 
their skewness assessed by the asymmetry function $\gamma(p),$ $0<p<1$, 
of Critchley and Jones \cite{critchley-jones} and 
provide relatively rare tractable 
examples thereof. Jones \cite{jones1} shows that the asymmetry function 
associated with unimodal $f_b$ is
\begin{equation}
\gamma_b(p) = \left( \sqrt{c_g^2(p)+4b} 
-\sqrt{4b}\right) /  c_g(p) 
\label{12.4}
\end{equation}
\noindent
where $c_g(p) = g^{-1}(pg(0))$. 
This shows that the Cauchy-Schl\"omilch transformation of scale 
always results in positively skewed distributions, with asymmetry 
functions decreasing in $p$, 
which become more skew 
as $b$ decreases. For example, for the RRIG density (\ref{int-11}), 
\begin{equation}
\gamma_b(p) =  \left( \sqrt{2b-\log p} 
-\sqrt{2b}\right)/\sqrt{-\log p}. 
\label{RRIG-for}
\end{equation}
\noindent
The extended Cauchy-Schl\"omilch transformation described in Section 11 
also affords new transformation of scale distributions, as explored by 
Jones \cite{jones1}. Probability densities of the form
\begin{equation}
f_s(x) = g(|x-s(x)|) 
\label{12.5}
\end{equation}
\noindent
for decreasing, onto, self-inverse $s$ are described there. In this 
situation, properties (i)-(v) discussed above become:
\smallskip

\noindent (i$^{\prime}$) $f_s$ can be defined to be S-symmetric about 
S-center 
${x_0}$ since $f_s(x) = f_s(s(x))$. Here $x_0$ is defined by $s(x_0) 
= x_0$; 
\smallskip

\smallskip
\noindent (ii$^\prime$) $f_s(0)=0$ and 
$$ f_s(x) \approx g(s(x)) ~{\rm as}~x \rightarrow 0;~~~~~
f_s(x) \approx g(x) ~{\rm as}~x \rightarrow \infty ;$$

\noindent (iii$^\prime$) moment relationships follow from
$$E_{f_s} \{ |X-s(X)|^r\} = E_g(Y^r)$$
and, by S-symmetry,
$$E_{f_s}(X^r) = - E_{f_s}(s'(X)s^r(X) ).$$
A special case of the latter is that $E_{f_s}(s'(X))=-1$; 

\smallskip
\noindent (iv$^\prime$) if $g$ is decreasing, then $f_s$ is unimodal with 
mode at $x_0$;

\smallskip
\noindent (v$^\prime$) if $f_s$ is unimodal and $g$ is convex, its mean 
and its median are 
both greater than its mode. \\

By way of example, Jones \cite{jones1} briefly explored the half-Gaussian-based 
analogue of (\ref{int-11}) when $s(x)$ is given by (\ref{formula-fors}). This 
has probability density
\begin{equation}
f_s(x) =  \sqrt{\frac{2}{\pi}}\exp\left\{ -\frac1{2\alpha^2} 
\log^2 \left( e^{\alpha x}- 1\right) \right\}.
\label{12.6}
\end{equation}
\noindent
The fact that (\ref{12.6}) integrates to $1$ is closely related to 
(\ref{jones-11}). 
Its asymmetry function has the form
\begin{equation}
\gamma_s(p) =  \frac1{\alpha \sqrt{-(\log p)/2}} \log \left\{ \cosh 
\left( {\alpha \sqrt{-(\log p)/2}} \right)\right\}. 
\label{12.7}
\end{equation}
\noindent
Like (\ref{RRIG-for}), this 
asymmetry function is always positive and decreases in 
$p$; (\ref{12.7}) increases in $\alpha$.

The Cauchy-Schl\"olmilch transformation has thus motivated and 
triggers a new and promising area of work in distribution theory.

\section{Conclusions}
\label{sec-conclusions}
\setcounter{equation}{0}

The Cauchy-Schl\"omilch transformation establishes the equality of two 
definite integrals with  integrands related in a simple manner. Applying 
this transformation to a variety of well-known definite integrals yields
examples that are beyond the current capabilities of symbolic languages. 
Our purpose in this paper is not only to present the many integrals 
considered here, but also to give an exposition of the salient points of 
the Cauhy-Schl\"omilch transformation so as to serve as 
motivating examples to explore further 
symbolic integration algorithms.  \\

\noindent
{\bf Acknowledgements}. The fourth author acknowledges the partial 
support of 

\noindent
$\text{NSF-DMS } 0713836$. 

%

\end{document}

\end{document}